\documentclass[a4paper,12pt]{amsart}
\usepackage{amsmath, amssymb,braket}
\usepackage{amsthm}
\usepackage[vcentermath,enableskew]{youngtab}
\usepackage{bm}

\usepackage{url}

\newcommand{\risaasir}{{\sf risa/asir}}
\newcommand{\AAA}{{\mathcal A}}

\newcommand{\kk}{\mathbb{K}}
\newcommand{\ZZ}{\mathbb{Z}}

\newcommand{\NN}{\mathbb{N}}

\newcommand{\der}{\partial}

\newcommand{\numof}[1]{\left| #1 \right|}

\newcommand{\Oru}[3]{\substack{#1 \\ #2 \\ #3}}

\newcommand{\supp}{\operatorname{supp}}

\newcommand{\thirdalgo}{\operatorname{\tt ALG3}}
\newcommand{\secondalgo}{\operatorname{\tt ALG2}}
\newcommand{\firstalgo}{\operatorname{\tt ALG1}}

\newcommand{\defit}[1]{{\em #1}}

\theoremstyle{plain}
\newtheorem{thm}{Theorem}[section]
\newtheorem{lemma}[thm]{Lemma}
\newtheorem{cor}[thm]{Corollary}
\newtheorem{theorem}[thm]{Theorem}
\newtheorem{proposition}[thm]{Propsition}

\theoremstyle{remark}
\newtheorem{remark}[thm]{Remark}

\theoremstyle{definition}
\newtheorem{definition}[thm]{Definition}

\allowdisplaybreaks[3]

\begin{document}

\title[An algorithm to construct basis]{
An Algorithm to Construct A Basis for the Module of Logarithmic Vector Fields
}
\author[Numata,y.]{Numata, yasuhide}

\begin{abstract}
We consider logarithmic vector fields 
parametrized by finite collections of weighted hyperplanes.
For a finite collection of weighted hyperplanes in a two-dimensional
 vector space, it is known that
the set of such vector fields is
a  free module of rank two whose basis elements are
homogeneous.
We give an algorithm to construct a homogeneous basis for the module.
\end{abstract}

\maketitle

\section{Introduction}
For an $l$-dimensional vector space  $V$,
a finite collection $\AAA$ of subspaces of  $V$  whose codimensions
are one is called \defit{a central hyperplane arrangement} in $V$.
A pair $(\AAA, \mu)$ of a central hyperplane arrangement in $V$
and a map $\mu:\AAA \to \ZZ_{>0}$ is called a \defit{multiarrangement} 
in $V$. 
Let $V^{\ast}$ be the dual space of $V$, and $S$ 
the algebra  of polynomial functions on $V$.
The algebra $S$ has a natural graded structure.
For a multiarrangement $(\AAA,\mu)$ in $V$,
we define $D(\AAA,\mu)$ to be
the set of derivations $\theta$ of $S$ satisfying the following condition:
$\theta(\alpha)$ is in the ideal generated by $\alpha^{\mu(\ker(\alpha))}$
for each $\alpha \in S$ with $\ker(\alpha)\in \AAA$.
The set  $D(\AAA,\mu)$  also has a natural structure of a graded $S$-module.
We say that a multiarrangement $(\AAA,\mu)$ is free if 
the corresponding module $D(\AAA, \mu)$ is free.
For a free multiarrangement  $(\AAA,\mu)$, 
 $D(\AAA,\mu)$  has a homogeneous basis.
The multi-set of degrees of a homogeneous basis for  $D(\AAA,\mu)$ 
is  called the exponents of $(\AAA,\mu)$. 

Ziegler showed that $D(\AAA,\mu)$ is free for 
a multiarrangement $(\AAA,\mu)$ in  a two-dimensional vector space \cite{z}.
Yoshinaga proved a theorem which characterizes the free
multiarrangements in a
three-dimensional vector space.
In the theorem, exponents of multiarrangements in  a two-dimensional
vector space play an important role \cite{y}.
For some multiarrangements $(\AAA,\mu)$ in  a two-dimensional vector space,
homogeneous bases for $D(\AAA,\mu)$ are explicitly given.
Hence exponents for these multiarrangements are also 
 explicitly given. (See \cite{st, w} and so on.)  

The set of multiarrangements has a structure of a graded poset. 
(See Section \ref{notations}.)
The mapping $(\AAA,\mu)$ to $D(\AAA,\mu)$ is order-reversing.
In this paper, for multiarrangements $(\AAA,\mu)$ and  $(\AAA,\mu')$ in  a
two-dimensional vector space 
such that $(\AAA,\mu')$ is larger than $(\AAA,\mu)$,
we show an algorithm to construct
a homogeneous basis for $D(\AAA,\mu')$ from
a homogeneous basis for $D(\AAA,\mu)$.

In Section \ref{notations}, we recall the definition
of the module of logarithmic vector fields and K. Saito's criterion.
In Section \ref{mainsec}, we show algorithms to construct a
homogeneous basis for $D(\AAA,\mu)$
for a multiarrangement $(\AAA,\mu)$ in  a
two-dimensional vector space.
In Section \ref{applysec}, we show some applications of our algorithms.
In Subsection \ref{applysec:dif}, we consider
rise and fall of the difference of exponents.
In Subsection \ref{applysec:finite}, 
we consider the case where the base field is finite, 
and explicitly describe 
a homogeneous basis for some $D(\AAA,\mu)$.
In Subsection \ref{applysec:implementation}, 
 an implementation of our algorithms as a program of the computer
algebra system \risaasir{} is outlined.

\section{Definition and Notation}\label{notations}
Let $\kk$ be a field, $V$ a two-dimensional vector space over $\kk$,
and $S$ a polynomial ring $\kk[x,y]$.
The algebra $S$ naturally has a structure $S=\bigoplus_{i\in\NN} S_i$
of a graded algebra, where $S_i$ is a vector space whose basis
is $\Set{x^j y^{i-j}| j=0,\ldots,i}$.

We call a finite collection $\AAA$ of hyperplanes containing the origin
in $V$ a \defit{central arrangement}.
A pair $(\AAA,\mu)$ of a central arrangement $\AAA$ and a map 
$\mu$ from $\AAA$ to the set of positive integers $\ZZ_{>0}$
is called  \defit{a multiarrangement}.
We identify a multiarrangement with
a map $\mu$ from the set of  hyperplanes in $V$ 
containing the origin
to the set of nonnegative integers $\NN$ 
whose support $\supp(\mu)=\Set{H | \mu(H) \neq 0}$ is a finite set.
We write $\numof{\mu}$ for $\sum_{H} \mu(H)$.
We define $(\AAA, \mu) \subset (\AAA', \mu')$ 
if $\mu$ and $\mu'$ satisfy $\mu(H) \leq \mu'(H)$ for all $H \in \AAA$.
The set of multiarrangements has
a structure of a graded poset with the minimum element $(\emptyset, 0)$
by the relation $\subset$, where 
$0$ is the map such that
$0(H)=0$ for all $H$.

\begin{definition}
For a multiarrangement $(\AAA,\mu)$,
we define the set
$D(\AAA,\mu)$ of logarithmic vector fields to be
\begin{gather*}
\Set{ \theta = f(x,y) \der_x + g(x,y) \der_y | 
\Oru{\text{$f(x,y),g(x,y)\in S.$}}{\text{$\theta(\alpha)$ is in the
 ideal $\left(\alpha^{\mu(\ker(\alpha))}\right) $ }}{\text{for each  
$\alpha \in S_{1}\setminus\Set{0}$.}}},
\end{gather*}
where $\der_x$ and $\der_y$ respectively denote the partial differential
 operators in the variables $x$ and $y$. 
\end{definition}

For a multiarrangement $(\AAA,\mu)$,
$D(\AAA,\mu)$ has 
a structure $D(\AAA,\mu)=\bigoplus_{i\in \NN} D(\AAA,\mu)_i$ of 
a graded 
 $S$-module, where 
\begin{gather*}
D(\AAA,\mu)_i = \Set{  f(x,y) \der_x + g(x,y) \der_y \in D(\AAA,\mu)| 
f(x,y), g(x,y) \in S_i }.
\end{gather*}

The mapping $(\AAA,\mu) \mapsto D(\AAA,\mu)$
is order-reversing.
Namely,
by definition,
$D(\AAA,\mu) \subset D(\AAA',\mu')$
for multiarrangements $(\AAA',\mu') \subset (\AAA,\mu)$.

We consider 
only the case of two-dimensional vector spaces in this paper.
In this case, it is known that
$D(\AAA,\mu)$ is always  a free $S$-module whose rank is two, and that
$D(\AAA,\mu)$ has a homogeneous basis.
We give an algorithm to construct a basis for $D(\AAA,\mu)$ in this paper. 
The following  is a well-known criterion.

\begin{theorem}[K. Saito's criterion \cite{s}]\label{saitoscriterion}
Let $\Set{\theta_i \in D(\AAA,\mu)_{d_i}|i=1,2}$ 
be $S$-linearly independent.
Then $\Set{\theta_1,\theta_2}$
is a basis for $D(\AAA,\mu)$
if and only if
$\numof{\mu} = \sum_i d_i$.
\end{theorem}

\section{Main result}\label{mainsec}
In this section, we give  algorithms
 to construct a homogeneous basis for $D(\AAA,\mu)$.

First we consider the case where
 $\mu \subset \mu'$ and $\numof{\mu'}=\numof{\mu}+1$.
We show an algorithm to construct a homogeneous basis for
$D(\AAA',\mu')$ from 
a homogeneous basis for $D(\AAA,\mu)$.

\begin{theorem}\label{thm:algorithm1}
Let $(\AAA,\mu)$ be a multiarrangement. 
Fix $\alpha=\alpha_x x + \alpha_y y \in S_{1}\setminus\Set{0}$. 
Let $\AAA'$ be $\AAA \cup \ker(\alpha)$,
and $\mu'$ the map such that 
\begin{gather*}
\mu'(H)=
\begin{cases}
\mu(\ker(\alpha)) + 1 & (H=\ker(\alpha))\\
\mu(H) & ( H \in \AAA \setminus \Set{\ker(\alpha)}). 
\end{cases}
\end{gather*}

Then we can construct a homogeneous basis $(\theta'_1, \theta'_2)$ 
for $D(\AAA',\mu')$
from a homogeneous basis $(\theta_1,\theta_2)$ for $D(\AAA,\mu)$ by
the following algorithm $\firstalgo:$ 
\begin{description}
\item[Input] $\theta_1$, $\theta_2$, 
$\alpha=(\alpha_x x +\alpha_y y)$, $m=\mu(\ker(\alpha))$.
\item[Output] $\theta'_1$, $\theta'_2$.
\item[Procedure] \ 
\begin{enumerate} 
\item If $\deg(\theta_1) < \deg(\theta_2)$, 
then swap $\theta_1$ and $\theta_2$. 
\item Let $g(x,y)=\frac{\theta_2(\alpha)}{\alpha^{m}}$.
\item If $g(\alpha_y,-\alpha_x)=0$, then
let $\theta'_1=\alpha\cdot \theta_1$ and $\theta'_2=\theta_2$,
and finish the procedure.
\item Let $f(x,y)=\frac{\theta_1(\alpha)}{\alpha^{m}}$. 
\item If $f(\alpha_y,-\alpha_x)=0$, then
let $\theta'_1=\theta_1$ and $\theta'_2=\alpha\cdot \theta_2$,
and finish the procedure.
\item \label{alg1:step6} Let $q(x,y)$ be a homogeneous polynomial in variables $x$, $y$ of
 degree $\deg(\theta_1)-\deg(\theta_2)$ satisfying
the equation
\begin{gather*}
f(\alpha_y,-\alpha_x)+g(\alpha_y,-\alpha_x) q(\alpha_y,-\alpha_x)=0.
\end{gather*}
\item Let $\theta'_1=\theta_1+q(x,y)\cdot \theta_2$ and $\theta'_2=\alpha\cdot \theta_2$.
\end{enumerate}
\end{description}
\end{theorem}

\begin{proof}
Since $\theta_1$, $\theta_2 \in D(\AAA,\mu)$,
both
$f(x,y)=\frac{\theta_1(\alpha)}{\alpha^{m}}$ and
$g(x,y)=\frac{\theta_2(\alpha)}{\alpha^{m}}$
are polynomials.

First we consider the case  $g(\alpha_y,-\alpha_x)=0$.
Since $\theta_2$ is homogeneous, 
$g(\alpha_y,-\alpha_x)$ is also homogeneous.
Since $\kk$ is a field,
$g(\alpha_y,-\alpha_x)=0$
if and only if  $g(x,y)$ is divisible by $\alpha$. 
This implies $\theta_2(\alpha) \in \alpha^{m+1} S$.
Hence $\theta_2$ is in $D(\AAA',\mu')$.
Since $\theta_1 \in D(\AAA,m)$,  $\alpha\cdot \theta_1$ is also in 
$D(\AAA',\mu')$.
Since $\theta_1$ and $\theta_2$ are linearly independent,
 $\alpha\cdot \theta_1$ and $\theta_2$ are linearly independent. 
It is clear that
$\deg(\alpha\cdot \theta_1)+\deg(\theta_2)=\deg(\theta_1)+\deg(\theta_2)+1$.
Hence $(\alpha\cdot \theta_1,\theta_2)$ is a basis for $D(\AAA',\mu')$
by Theorem \ref{saitoscriterion}.

Next we consider the case
$g(\alpha_y,-\alpha_x)\neq 0$.
Let $d=\deg(\theta_1)-\deg(\theta_2)$.
It follows from $g(\alpha_y,-\alpha_x)\neq 0$ that the equation
\begin{gather*} 
f(\alpha_y,-\alpha_x)+g(\alpha_y,-\alpha_x) \sum_{i=0}^{d}q_i\alpha_y^{i}(-\alpha_x)^{d-i}=0
\end{gather*}
is solvable. For example, 
\begin{gather}
\begin{cases}
q_0=-\frac{f(\alpha_y,-\alpha_x)}{g(\alpha_y,-\alpha_x)(-\alpha_x)^{d}}
 -\sum_{i=1}^{d}\frac{\alpha_y^{i}}{(-\alpha_x)^{i}} & (i=0)\\
q_i=1  & (i>0)
\end{cases}
\label{defofq}
\end{gather}
is one of solutions if $\alpha_x\neq 0$.
For such $q_i$, let $q(x,y)=\sum_{i=0}^{d} q_i x^{i} y^{d-i}$ and
$\theta'_2=\theta_1+q(x,y)\cdot \theta_2$.
Since
$\theta'_1(\alpha)=\theta_1(\alpha)+q(x,y)\cdot \theta_2(\alpha)$,
\begin{align*}
\frac{\theta'_1(\alpha)}{\alpha^m}&=\frac{\theta_1(\alpha)}{\alpha^m}+\frac{q(x,y)\cdot \theta_2(\alpha)}{\alpha^m}\\
&=f(x,y)+q(x,y)g(x,y).
\end{align*}
Since $\frac{\theta'_1(\alpha)}{\alpha^m}$ is homogeneous, and
$f(\alpha_y,-\alpha_x)+q(\alpha_y,-\alpha_x)g(\alpha_y,-\alpha_x)=0$,
$\frac{\theta'_1(\alpha)}{\alpha^m}$ is divisible by $\alpha$.
Hence we have $\theta'_1 \in D(\AAA',\mu')$.
Since $\theta_2\in D(\AAA,\mu)$,
 $\alpha\cdot \theta_2 \in D(\AAA',\mu')$.
The linearly independence of
$\Set{\theta_1, \theta_2}$ implies 
the linearly independence of
$\Set{\theta'_2=\theta_1+q(x,y)\cdot \theta_2, \alpha\cdot \theta_2}$.
It is clear that
$\deg(\theta'_1)+\deg(\alpha\cdot\theta_2)=\deg(\theta_1)+\deg(\theta_2)+1$.
Hence $(\theta'_1,\alpha\cdot \theta_2)$ is a basis for $D(\AAA',\mu')$.
\end{proof}

By taking the polynomial defined by (\ref{defofq}) 
as $q(x,y)$ in Step (\ref{alg1:step6}) of the algorithm 
$\firstalgo$,
we have the following algorithm.

\begin{cor}\label{thm:algorithm2}
Let $(\AAA,\mu)$ be a multiarrangement. 
Fix $\alpha=\alpha_x x + \alpha_y y \in S_{1}\setminus\Set{0}$. 
Let $\AAA'$ be $\AAA \cup \ker(\alpha)$,
and $\mu'$ the map such that 
\begin{gather*}
\mu'(H)=
\begin{cases}
\mu(\ker(\alpha)) + 1 & (H=\ker(\alpha))\\
\mu(H) & ( H \in \AAA \setminus \Set{\ker(\alpha)}). 
\end{cases}
\end{gather*}

Then we can construct a homogeneous basis $(\theta'_1, \theta'_2)$ 
for $D(\AAA',\mu')$
from a homogeneous basis $(\theta_1,\theta_2)$ for $D(\AAA,\mu)$ by
the following algorithm $\secondalgo :$
\begin{description}
\item[Input] $\theta_1$, $\theta_2$, 
$\alpha=(\alpha_x x +\alpha_y y)$, $m=\mu(\ker(\alpha))$.
\item[Output] $\theta'_1$, $\theta'_2$.
\item[Procedure]
\ 

\begin{enumerate} 
\item If $\deg(\theta_1) < \deg(\theta_2)$, 
then swap $\theta_1$ and $\theta_2$. 
\item Let $g(x,y)=\frac{\theta_2(\alpha)}{\alpha^{m}}$.
\item If $g(\alpha_y,-\alpha_x)=0$, then
let $\theta'_1=\alpha\cdot \theta_1$ and $\theta'_2=\theta_2$,
and finish the procedure.
\item Let $f(x,y)=\frac{\theta_1(\alpha)}{\alpha^{m}}$. 
\item If $f(\alpha_y,-\alpha_x)=0$, then
let $\theta'_1=\theta_1$ and $\theta'_2=\alpha\cdot \theta_2$,
and finish the procedure.
\item Let $d=\deg(\theta_1)-\deg(\theta_2)$.
\item If $\alpha_x=0$, then
let $\theta'_1=\theta_1-\frac{f(1,0)}{g(1,0)}x^d \theta_2$ 
and $\theta'_2=y \cdot \theta_2$,
and finish the procedure.
\item Let $q(x,y)$ be 
\begin{gather*}
\left(-\frac{f(\alpha_y,-\alpha_x)}{g(\alpha_y,-\alpha_x)(-\alpha_x)^{d}}
 -\sum_{i=1}^{d}\frac{\alpha_y^{i}}{(-\alpha_x)^{i}} \right) y^d
+ \sum_{i=1}^{d} x^i y^{d-i}.
\end{gather*}
\item Let $\theta'_1=\theta_1+q(x,y)\cdot \theta_2$ and 
$\theta'_2=\alpha\cdot \theta_2$.
\end{enumerate}

\end{description}
\end{cor}

We have a basis $(\der_x,\der_y)$ for $D(\emptyset,0)$.
By applying the algorithm in Theorem \ref{thm:algorithm1} recursively,
we can construct a basis for $D(\AAA,\mu)$ for all multiarrangements.

\begin{theorem}\label{thm:algorithm3}
We can construct a homogeneous basis $(\theta_1, \theta_2)$ 
for a multiarrangement $D(\AAA,\mu)$ by
the following algorithm $\thirdalgo:$
\begin{description}
\item[Input] $(\AAA,\mu)$.
\item[Output] $(\theta_1,\theta_2)$.
\item[Procedure]
\ 
\begin{enumerate}
\item If $\AAA = \emptyset$, then let $\theta_1=\der_x$ and  $\theta_2=\der_y$,
and finish the procedure.
\item Let $H$ be a hyperplane in $\supp(\mu)$.
\item Let $\mu'$ be a map such that
\begin{gather*}
\mu'(H')=
\begin{cases}
\mu(H)-1 & (H'=H), \\
\mu(H') &(H' \in \AAA \setminus\Set{H}).
\end{cases}
\end{gather*}
\item Let $\AAA' = \supp(\mu')$.
\item Let $(\theta'_1, \theta'_2)$ be the resulting basis of
	   $\thirdalgo(\AAA',\mu')$.
\item Let $\alpha\in S_1$ be a linear form such that $\ker(\alpha)=H$.
\item Let $(\theta_1, \theta_2)$ be the resulting basis of
	   $\firstalgo(\theta'_1, \theta'_2,\alpha,\mu'(H))$.
\end{enumerate}
\end{description}
\end{theorem}

\section{Application}\label{applysec}
In this section, we show  applications of our algorithms.
First we consider the difference between the degrees of 
elements of homogeneous basis. 
Next we explicitly 
 describe a homogeneous basis for some arrangements 
in the case where $\kk$ is a finite field.
Finally we  briefly introduce an implementation of our algorithms.

\subsection{Difference between exponents}\label{applysec:dif}
In this subsection, we consider the difference between the degrees of 
elements of homogeneous basis, i.e., the difference between the
exponents  of a multiarrangement in two-dimensional vector space. 
When a multiarrangement is  made larger,
the difference of its exponents either increases or decreases. 
We show that
the difference decreases if a generic hyperplane is added to 
a multiarrangement with the two degrees different.

The next two corollaries
follow from Theorem \ref{thm:algorithm1}.
\begin{cor}
\label{cor:difI}
Let $(\AAA,\mu)$ be a multiarrangement,
and $\Set{\theta_1,\theta_2}$ a homogeneous basis for $D(\AAA,\mu)$
 such that $\deg(\theta_1) \geq \deg(\theta_2)$.

Fix $\alpha=\alpha_x x+\alpha_y y\in S_{1}\setminus\Set{0}$.
Let $\mu'$ be the map such that 
\begin{gather*}
\mu'(H)=
\begin{cases}
\mu(\ker(\alpha))+1 &(H=\ker(\alpha)) \\
\mu(H) & (H\in\AAA\setminus\Set{\ker(\alpha)}).
\end{cases}
\end{gather*}
Let $\Set{\theta'_1,\theta'_2}$ be a homogeneous basis for $D(\supp{\mu'},\mu')$.

Let  $d'=\numof{\deg(\theta'_1) - \deg(\theta'_2)}$, 
$d=\numof{\deg(\theta_1) - \deg(\theta_2)}$. Then 
\begin{gather*}
\begin{cases}
d'>d
& \text{($g(\alpha_y,-\alpha_x)=0$ or
$d=0$)} \\
d'<d
& \text{(otherwise)}, 
\end{cases}
\end{gather*}
where $g = \frac{\theta_2(\alpha)}{ \alpha^{\mu(\ker(\alpha))} }$.
\end{cor}

\begin{proof}
The corollary directly follows from the algorithm $\firstalgo$
in Theorem \ref{thm:algorithm1}. 
\end{proof}

\begin{cor}
\label{cor:difII}
Let $(\AAA,\mu)$ be a multiarrangement,
and $\Set{\theta_1,\theta_2}$ a homogeneous basis for $D(\AAA,\mu)$ such
 that $\deg(\theta_1) \geq \deg(\theta_2)$.
Let $\theta_2= \varphi(x,y)\der_x + \psi(x,y)\der_y$.
For $\alpha_x$, $\alpha_y$ satisfying 
\begin{gather*}
\alpha_x \varphi(\alpha_y, -\alpha_x) + \alpha_y \psi(\alpha_y,-\alpha_x) \neq 0,
\end{gather*}
let $\alpha = \alpha_x x + \alpha_y y$,
$\AAA'=\AAA\cup\Set{\ker(\alpha)}$,
and  $(\AAA', \mu')$  a multiarrangement such that 
\begin{gather*}
\mu'(H)=
\begin{cases}
1       &(H=\ker(\alpha))\\
\mu(H)  &(H\in\AAA\setminus\Set{\ker(\alpha)}).
\end{cases}
\end{gather*}

If $\deg(\theta_1) - \deg(\theta_2) = 1$,
then
$\deg(\theta_1') = \deg(\theta_2')$,
where $\Set{\theta'_1,\theta'_2}$ is a homogeneous basis for 
$D(\AAA',\mu')$.
\end{cor}

\begin{proof}
Since 
\begin{gather*}
\theta_2(\alpha_x x+ \alpha_y y) |_{x=\alpha_y, y=-\alpha_x} =
\alpha_x \varphi(\alpha_y, -\alpha_x) + \alpha_y \psi(\alpha_y,-\alpha_x) \neq 0,
\end{gather*}
$\theta_2(\alpha)$ does not divisible by $\alpha=\alpha_x x+ \alpha_y y$.
Since $\theta_2\in D(\AAA,\mu)$, $\mu(\ker(\alpha))=0$.
Apply $\firstalgo$. Since 
\begin{gather*}
g(\alpha_x, -\alpha_y)=
\left.\frac{\theta_2(\alpha_x x+ \alpha_y y)}{\alpha^0}
 \right|_{x=\alpha_y, y=-\alpha_x} \neq 0, 
\end{gather*}
\begin{gather*}
|\deg(\theta_1') - \deg(\theta_2')| < |\deg(\theta_1) - \deg(\theta_2)|
\end{gather*}
if
$\deg(\theta_1) > \deg(\theta_2)$,
where $\Set{\theta'_1,\theta'_2}$ is 
the homogeneous basis for $D(\AAA',\mu')$
obtained from $(\theta_1, \theta_2, \alpha, 0)$ 
by the algorithm $\firstalgo$.
Hence we have the corollary.
\end{proof}

Corollary \ref{cor:difI} implies the following corollary.
\begin{cor}
Fix $H=\ker(\alpha)\in\AAA$. Let $\mu$ satisfy 
\begin{gather*}
2 \mu(H) > \numof{\mu} .
\end{gather*}
If $\Set{\theta_1,\theta_2}$ is a basis for $D(\AAA,\mu)$ satisfying 
$\deg(\theta_1)>\deg(\theta_2)$, then
$\Set{\alpha^{n}\theta_1,\theta_2}$ is a basis for
 $D(\AAA,\mu'')$, 
where $\mu''$ is the map such that 
\begin{gather*}
\mu''(H')=
\begin{cases}
\mu(H)+n & (H'=H) \\
\mu(H') &(H'\in\AAA\setminus\Set{H}).
\end{cases}
\end{gather*}
\end{cor}
\begin{proof}
It is enough to show the case where $n=1$.
In this case,  $2 \mu'(H) > \numof{\mu'}$, $2 \mu(H) > \numof{\mu}$, 
and $\mu'(H) = \mu(H)+1$.
For a multiarrangement $(\AAA, \kappa)$ such that 
 $2 \kappa(H) \geq \numof{\kappa}$ for some $H\in \AAA$,  
it is known that the exponents of $(\AAA,\kappa)$
are $(\kappa(H), \numof{\kappa}-\kappa(H))$.
(See \cite{ot}.)
Since $2 \mu(H) > \numof{\mu}$ and $2 \mu'(H) > \numof{\mu'}$,
we have 
\begin{align*}
\deg(\theta_1)&=\mu(H),\\
\deg(\theta_2)&=\numof{\mu}-\mu(H),\\
\deg(\theta'_1)&=\mu'(H) =\mu(H)+1,\\
\deg(\theta'_2)&=\numof{\mu'}-\mu'(H)=\numof{\mu}-\mu(H).
\end{align*}
Hence  
$\deg(\theta'_1)-\deg(\theta'_2)
=2\mu(H)-\numof{\mu} +1
=\deg(\theta_1)-\deg(\theta_2)+1.$
Since the difference increases,  
$g(\alpha_y,-\alpha_x)=0$ by Corollary \ref{cor:difI}.
Since $g(\alpha_y,-\alpha_x)=0$,
$\Set{\alpha\theta_1,\theta_2}$ is a basis for
 $D(\AAA,\mu'')$ by Theorem \ref{thm:algorithm1}.
Hence we have the corollary.
\end{proof}

\subsection{Finite Fields}\label{applysec:finite}
In this subsection, let $p$ be a prime number.
Let us take the finite field $F_q$  consisting of $q=p^n$ elements
as $\kk$.
Let $\AAA$ be the set of hyperplanes in $F_q^2$.
We explicitly describe homogeneous bases for some multiarrangements
 $D(\AAA,\mu)$.

Let $\theta_{q^i}$ be $x^{q^i} \der_x + y^{q^i} \der_y$,
and $\kappa_{q^i}$  the map such that $\kappa_{q^i}(H)=q^i$ for all hyperplanes $H$.

\begin{lemma}\label{finite:specialcase}
The set 
$\Set{\theta_{q^{i}},\theta_{q^{i+1}}}$
is 
a basis for $D(\AAA,\kappa_{p^i})$.
\end{lemma}
\begin{proof}
Since $(a x + b y)^{q}= a x^q + b y^q$,
$\theta_{q^i}$ and $\theta_{q^{i+1}} $
are in $D(\AAA,\kappa_{q^i})$.
Since $x^{q^i} y^{q^{i+1}} - x^{q^{i+1}} y^{q^i} \neq 0$, 
$\Set{\theta_{q^i},\theta_{q^{i+1}}}$
is linearly independent.
Since the number $\numof{\AAA}$ of hyperplanes in $F_q^2$ is
 $ \frac{\numof{F_q^2\setminus{(0,0)}}}{F_q^{\times}} =\frac{q^2-1}{q-1}=q+1$, 
 $|\kappa_{q^i}| = q^{i}(q+1) = \deg(\theta_{q^{i}} )+\deg(\theta_{q^{i+1}} )$.
Hence
the set 
$\Set{\theta_{q^{i}},\theta_{q^{i+1}}}$
is 
a basis for $D(\AAA,\kappa_{q^i})$.
\end{proof}

\begin{cor}\label{finite:main}
Let $i$ be an positive integer.
For each $H \in \AAA$, fix $\alpha_H(x,y) \in S_{1}$ 
satisfying $\ker(\alpha_H(x,y))=H$,
and let $j_H$ be an integer such that $0 \leq j_H \leq q^{i+1}-q^i$.
Let $\alpha = \prod_{H} \alpha_H^{j_H}$. 
Let $\mu$ be the map from $\AAA$ to $\ZZ_{>0}$ such that $\mu(H)=q^i+j_H$ for
 each hyperplane $H$.

Then
\begin{gather*}
\Set{\alpha(x,y)\cdot \theta_{q^{i}}, \theta_{q^{i+1}}}
\end{gather*}
is a basis for $D(\AAA,\mu)$.
\end{cor}

\begin{proof}
Fix a saturated chain 
\begin{gather*}
((\AAA,\mu_{0})=(\AAA,\kappa_{q^i}),(\AAA,\mu_{1}),\ldots,(\AAA,\mu_{n})=(\AAA,\mu))
\end{gather*}
of multiarrangements. 
For $0\leq k \leq n$,
let $\alpha_{\mu_k}=\prod_{H\in\AAA} \alpha_H^{\mu_{k}(H)-\mu_{0}(H)}$,
and $\theta'_{k}=\alpha_{\mu_k} \theta_{q^{i}}$.
For $0\leq k < n$,
let $\alpha_{k}=\frac{\alpha_{\mu_{k+1}}}{\alpha_{\mu_{k}}}$,
and $m_k=\mu_{k}(\ker(\alpha_{k}))$.
Since $\theta_{q^{i+1}}$ is in $D(\AAA,\mu_{k+1})$,
the resulting basis for $D(\AAA,\mu_{k+1})$ of the algorithm
$\secondalgo(\theta'_{k},\theta_{q^{i+1}},\alpha_{k},m_{k})$
is $\Set{\theta'_{k+1}, \theta_{q^{i+1}}}$
if $\Set{\theta'_{k},\theta_{p^{i+1}}}$ is 
a basis for $D(\AAA,\mu_{k})$.
Since the set $\Set{\theta_{q^{i}},\theta_{q^{i+1}}}$ is a basis for 
 $D(\AAA,\kappa_{q^i})=D(\AAA,\mu_{0})$,
we have the corollary.
\end{proof}

\subsection{Implementation}\label{applysec:implementation}
A simple implementation of our algorithms as 
a program for the computer algebra system \risaasir\cite{risaasir} 
is available in \cite{risaasirlibnu}.
The implementation is useful to compute exponents for many examples.
For example,
the following  can be observed with the implementation.
\begin{proposition}
Let 
$H_1={\Set{x+y=0}}$,
$H_2={\Set{x-y=0}}$,
$H_3={\Set{x=0}}$,
$H_4={\Set{y=0}}$,
and
$\AAA=\Set{H_i}$.
Let us assume that 
\begin{gather*}
\mu(H_{i})<\frac{\numof{\mu}}{2}
\end{gather*}
 for all $i$.
Then, for $20\leq \mu(H_{i}) \leq 30$,
$d=2$ if and only if
there exist $k,h,l\in \ZZ$ such that
\begin{align*}
 (\mu(H_1)&=2k+3+4h,  &\mu(H_2)&=2k+1,  & \mu(H_3)=\mu(H_4)&=2l), \\
 (\mu(H_3)&=2k+3+4h,  &\mu(H_4)&=2k+1,  & \mu(H_1)=\mu(H_2)&=2l), \\
 (\mu(H_1)&=2k+1+4h,  &\mu(H_2)&=2k+1,  & \mu(H_3)=\mu(H_4)&=2l+1),
\intertext{or}
 (\mu(H_3)&=2k+1+4h,  &\mu(H_4)&=2k+1,  & \mu(H_1)=\mu(H_2)&=2l+1) .
\end{align*}

\end{proposition}
\begin{remark}
For this computation,
we used \risaasir{} version 20050209 (Kobe Distribution) on Linux machine  
(CPU: Intel(R) Celeron(R) CPU 2.26GHz,
Memory: 494M, 
bogomips: 4521.98).
The $14641$ examples was computed in $40$ minutes.
\end{remark}

\subsection*{Acknowledgments}
The author would like to thank Professor Hiroaki Terao
for suggesting Lemma \ref{finite:specialcase} and 
 Corollary \ref{finite:main} for special cases.

\end{document}